\documentclass[12pt]{article}

\usepackage{mathtools, amsthm, amssymb, microtype}
\usepackage[margin=1in]{geometry}

\DeclareMathOperator{\Alg}{Alg}
\DeclareMathOperator{\Id}{Id}
\DeclareMathOperator{\im}{im}

\DeclarePairedDelimiter{\parens}{\lparen}{\rparen}
\DeclarePairedDelimiter{\braces}{\lbrace}{\rbrace}

\DeclarePairedDelimiter{\angs}{\langle}{\rangle}
\DeclarePairedDelimiterX\setc[2]{\lbrace}{\rbrace}{\,#1 \;\delimsize\vert\; #2\,}

\newcommand{\pid}[1]{\lparen #1 \rbrack}

\newcommand{\alg}[1]{\mathcal{A} \lparen #1 \rparen}
\newcommand{\mlat}[1]{\mathcal{L} \lparen #1 \rparen}

\theoremstyle{definition}
\newtheorem{defn}{Definition}[section]
\newtheorem{thm}[defn]{Theorem}
\newtheorem{prop}[defn]{Proposition}

\newtheorem{cor}[defn]{Corollary}

\title{Beth Definability in the Logic \textbf{KR}}
\author{Jacob Garber}
\begin{document}

\maketitle

\begin{abstract}
\noindent The Beth Definability Property holds for an algebraizable logic if and only if every epimorphism in the corresponding category of algebras is surjective. Using this technique, Urquhart in 1999 showed that the Beth Definability Property fails for a wide class of relevant logics, including \textbf{T}, \textbf{E}, and \textbf{R}. However, the counterexample for those logics does not extend to the algebraic counterpart of the super relevant logic \textbf{KR}, the so-called Boolean monoids. Following a suggestion of Urquhart, we use modular lattices constructed by Freese to show that epimorphisms need not be surjective in a wide class of relation algebras. This class includes the Boolean monoids, and thus the Beth Definability Property fails for \textbf{KR}.

\end{abstract}

\section{Introduction}

Relevant logics were first introduced to avoid the \emph{paradoxes of material implication}, which are counterintuitive inferences that result from a mismatch between the intuitive meaning of implication and its formalization in classical logic. This work lead to the development of a wide swath of relevant logics, spanning from the basic logic \textbf{B} to the logic of relevant implication \textbf{R}, with many other relevant logics (such as ticket entailment \textbf{T} and relevant entailment \textbf{E}) in between. A comprehensive description of these logics and relevant logic in general can be found in any of \cite{anderson1975, anderson1992, dunn2002, routley1982, brady2003}. In this paper we will focus our attention on the logic \textbf{KR}, which consists of adding the paradoxical axiom \( (A \land \neg A) \rightarrow B \) to \textbf{R}. Despite being stronger than \textbf{R} and thus not a purely relevant logic, \textbf{KR} avoids various other paradoxes of material implication and does not collapse to classical logic (see Kerr \cite{kerr2019}). Following Dunn and Restall \cite{dunn2002}, we call any such logic a \emph{super relevant} logic. Several important properties of classical logic fail for the relevant logics, including the \emph{Beth Definability Property}.

\begin{defn}
    Let \( L \) be a propositional logic and \( \Sigma \) a set of formulas from \( L \) containing a variable \( p \). For a new variable \( q \), let \( \Sigma[p/q] \) denote the result of replacing all instances of \( p \) with \( q \). We say \( \Sigma \) \emph{implicitly defines} \( p \) if
    \begin{equation*}
        \Sigma \cup \Sigma[p/q] \vdash p \leftrightarrow q.
    \end{equation*}
    Alternatively, we say \( \Sigma \) \emph{explicitly defines} \( p \) if there is a formula \( A \) containing only the variables in \( \Sigma \) without \( p \), such that
    \begin{equation*}
        \Sigma \vdash p \leftrightarrow A.
    \end{equation*}
    A logic \( L \) is said to have the \emph{Beth Definability Property} if for any set of formulas \( \Sigma \) and variable \( p \), if \( \Sigma \) implicitly defines \( p \), then \( \Sigma \) also explicitly defines \( p \).
\end{defn}

The well-known Beth Definability Theorem states that the Beth Definability Property holds for classical propositional logic. However, as shown by Urquhart \cite{urquhart1999} in 1999 and extended by Blok and Hoogland \cite[Corollary 4.15]{blok2006} in 2006, the Beth Definability Property fails for all relevant logics between \textbf{B} and \textbf{R} (inclusive). The techniques of those papers rely on the fact that Boolean negation is implicitly (but not explicitly) definable in those logics. This approach does not extend to \textbf{KR}, where Boolean negation is identified with relevant negation and is thus explicitly definable. However, Urquhart conjectured that the Beth Definability Property fails for \textbf{KR} as well, and outlined a possible method of attack using algebraic logic.

In algebraic logic, every algebraizable logic \( L \) has a corresponding category of algebras \( \Alg L \). For instance, the algebras of classical logic are the Boolean algebras, and those of intuitionistic logic are the Heyting algebras. Using this correspondence, it is often possible to translate properties of a logic into properties of its corresponding algebra.
\begin{defn}
    For objects \( A, B \) in a category \( \mathcal{C} \), we say a map \( f : A \to B \) is an \emph{epimorphism} if for any other object \( C \) and maps \( g, h : B \to C \),
    \begin{equation*}
        \text{if } g \circ f = h \circ f, \text{ then } g = h.
    \end{equation*}
\end{defn}
The Beth Definability Property holds in an algebraizable logic \( L \) iff in the corresponding algebra \( \Alg L \) all epimorphisms are surjective (the \emph{ES} property). This correspondence was first proven by N\'{e}meti in \cite[Section 5.6]{henkin1985}, and further developed by Blok and Hoogland in \cite{blok2006}. Intuitively, one can think of epimorphisms as being implicit definitions, and surjections as being explicit ones. For algebras \( A \subseteq B \), we say \( A \) is an \emph{epic subalgebra} of \( B \) if the inclusion map \( i : A \to B \) is an epimorphism. Of course, this inclusion map is surjective iff \( A = B \). To disprove the Beth Definability Property for a logic, it thus suffices to find a proper epic subalgebra in its corresponding category. To apply this to \textbf{KR}, we will analyze epimorphisms in the category of algebras for \textbf{KR}, the \emph{Boolean monoids}. Boolean monoids are closely related to relation algebras, and can be equivalently defined as \emph{dense symmetric relation algebras}.

We will tackle this problem using the approach described in Urquhart \cite[Problem 5.3]{urquhart2017}. As discussed in that paper, there is a general correspondence between Boolean monoids and modular lattices. Every Boolean monoid contains a modular lattice, and given a modular lattice \( L \), one can construct a corresponding Boolean monoid \( \alg{L} \) that contains an isomorphic copy of \( L \) as a sublattice. As shown by Freese in \cite[Theorem 3.3]{freese1979}, there exist modular lattices \( A \) and \( B \) such that \( A \) is a proper epic sublattice of \( B \). Using the above correspondence, we extend this to the construction of a proper epic sub-Boolean monoid, which shows the Beth Definability Property fails for \textbf{KR}. This construction is rather general, and in fact shows that ES fails for a wide class of relation algebras that includes the Boolean monoids.

\section{Boolean Monoids and Modular Lattices}

As shown by Anderson and Belnap in \cite[Section 28.2.3]{anderson1975}, the algebraic counterpart of the logic \textbf{R} with truth constant \( t \) (sometimes denoted \( \textbf{R}^\textbf{t} \)) is the variety of \emph{De Morgan monoids}. These are De Morgan lattices with a commutative monoid operation. The addition of the axiom \( (A \land \neg A) \rightarrow B \) to \textbf{R} corresponds to adding the axiom \( a \land \neg a = 0 \) to the algebra, which reduces the De Morgan lattice to a Boolean algebra. Such objects, which we call \emph{Boolean monoids}, are the algebraic counterpart of \textbf{KR}. A particularly useful description of Boolean monoids is in terms of relation algebras. The following axiomatization of relation algebras is taken from Givant \cite[Definition 2.1]{introrel}.

\begin{defn}
    A \emph{relation algebra} is an algebra \( \angs*{A, \lor, \neg, \circ, ^{\smallsmile}, t} \) such that for all \( a, b, c \in A \),
    \begin{enumerate}
        \item \( a \lor b = b \lor a \)
        \item \( a \lor (b \lor c) = (a \lor b) \lor c \)
        \item \( \neg (\neg a \lor b) \lor \neg (\neg a \lor \neg b) = a \)
        \item \( a \circ (b \circ c) = (a \circ b) \circ c \)
        \item \( a \circ t = a \)
        \item \( a^{\smallsmile \smallsmile} = a \)
        \item \( (a \circ b)^{\smallsmile} = b^{\smallsmile} \circ a^{\smallsmile} \)
        \item \( \parens*{a \lor b} \circ c = \parens*{a \circ c} \lor \parens*{b \circ c} \)
        \item \( \parens*{a \lor b}^{\smallsmile} = a^{\smallsmile} \lor b^{\smallsmile} \)
        \item \( \parens*{a^{\smallsmile} \circ \neg (a \circ b)} \lor \neg b = \neg b \)
    \end{enumerate}
\end{defn}

With the standard definition of \( a \land b = \neg \parens*{\neg a \lor \neg b} \), axioms 1--3 imply that \( \angs*{A, \lor, \land, \neg} \) is a Boolean algebra, axioms 4--7 imply that \( \angs*{A, \circ, ^{\smallsmile}, t} \) is a monoid with involution, and axioms 8--10 relate the Boolean and monoid operations to each other.

For all \( a, b \in A \), a relation algebra \( A \) is called
\begin{enumerate}
    \item \emph{abelian} if \( a \circ b = b \circ a \),
    \item \emph{symmetric} if \( a^{\smallsmile} = a \),
    \item \emph{dense} if \( a \leq a \circ a \).
\end{enumerate}
In particular, a Boolean monoid can be equivalently defined as a \emph{dense symmetric relation algebra}.

One important result from the theory of relation algebras is that every abelian relation algebra contains a special set of elements that form a modular lattice.

\begin{defn}
    A lattice \( L \) is called \emph{modular} if for all \( x, y, z \in L \),
    \begin{equation*}
        x \leq z \implies x \lor (y \land z) = (x \lor y) \land z.
    \end{equation*}
    This implication is equivalent to the following dual identities:
    \begin{align*}
        \parens*{x \land y} \lor \parens*{x \land z} &= x \land \parens*{y \lor \parens*{x \land z}}, \\
        \parens*{x \lor y} \land \parens*{x \lor z} &= x \lor \parens*{y \land \parens*{x \lor z}}.
    \end{align*}
\end{defn}

\begin{defn}
    For a relation algebra \( A \), an element \( a \in A \) is called \emph{reflexive} if \( t \leq a \), \emph{symmetric} if \( a^{\smallsmile} = a \), and \emph{transitive} if \( a \circ a \leq a \). An element with all three of these properties is a \emph{reflexive equivalence} element. Define \( \mlat{A} \) to be the set of all reflexive equivalence elements of \( A \).
\end{defn}

\begin{thm}
    \label{thm:modularlat}
    For an abelian relation algebra \( A \), the set of reflexive equivalence elements \( \mlat{A} \) is closed under fusion and meet, and forms a bounded modular lattice. Join is given by \( a \circ b \), meet by \( a \land b \), \( t \) is the bottom element, and 1 is the top.
\end{thm}

\begin{proof}
    See Givant \cite[Corollary 5.17]{introrel}.
\end{proof}

We can also in some sense reverse the above theorem, and use a modular lattice to construct a relation algebra.

\begin{defn}
    A \textbf{KR} \emph{frame} or \emph{model structure} is a triple \( F = \angs{S, R, 0} \) of a set \( S \) with ternary relation \( R \) and distinguished element \( 0 \), satisfying:
    \begin{enumerate}
        \item \( R0ab \) iff \( a = b \)
        \item \( Raaa \)
        \item \( Rabc \) implies \( Rbac \) and \( Racb \) \qquad (total symmetry)
        \item \( Rabc \) and \( Rcde \) implies \( \exists f \in S \) such that \( Radf \) and \( Rfbe \) \qquad (Pasch's Postulate)
    \end{enumerate}
    The last condition has close ties to projective geometry and is explored by Urquhart in \cite{urquhart2017}.
\end{defn}

\begin{defn}
    For a \textbf{KR} frame \( F = \angs{S, R, 0} \), the \emph{complex algebra} of \( F \) is the algebra \( \alg{F} = \angs*{P(S), \cup, \cap, \, ^{c}, \circ, t} \), where
    \begin{enumerate}
        \item \( \angs*{P(S), \cup, \cap, \, ^{c}} \) is the Boolean algebra on the power set of \( S \).
        \item \( t = \lbrace 0 \rbrace \) is the monoid identity.
        \item For \( A, B \subseteq S \), fusion is defined as
    \begin{equation*}
        A \circ B = \setc{c \in S}{Rabc \text{ for some } a \in A, b \in B}.
    \end{equation*}
    \end{enumerate}
\end{defn}
We often write \( \alg{S} \) for the complex algebra when the ternary relation and distinguished element are understood from context.

\begin{thm}
    For a \textbf{KR} frame \( F \), the complex algebra \( \alg{F} \) is a Boolean monoid.
\end{thm}
\begin{proof}
    See Urquhart \cite[Section 2]{urquhart2017}.
\end{proof}

\begin{defn}
    For a lattice \( L \) with least element 0, we define the following ternary relation on the elements of \( L \):
    \begin{equation*}
        Rabc \iff a \lor b = a \lor c = b \lor c.
    \end{equation*}
    Then with this relation, \( \angs{L, R, 0} \) is a \textbf{KR} frame iff \( L \) is modular.
\end{defn}

\begin{proof}
    The first three properties of a \textbf{KR} frame follow immediately from the lattice structure of \( L \). The last property, Pasch's Postulate, is equivalent to the modular law on \( L \), which is shown in Urquhart \cite[Theorem 2.7]{urquhart2017}.
\end{proof}

Thus for a modular lattice \( L \) with 0, the \emph{lattice complex algebra} \( \alg{L} \) is a Boolean monoid. An alternative more direct proof of this can be found in Givant \cite[Section 3.7]{introrel}. The definition of a \textbf{KR} frame is also an instance of the more general notion of a \emph{relational structure}, where the construction of a complex algebra can be repeated in the context of Boolean algebras with operators. Givant \cite[Chapter 19]{advancedrel} and \cite[Chapter 1]{dualitybao} go into more detail.

For a lattice complex algebra \( \alg{L} \), a particularly simple description of its reflexive equivalence elements can be given in terms of the ideals of \( L \).

\begin{defn}
    For a lattice \( L \), an \emph{ideal} of \( L \) is a non-empty subset \( J \subseteq L \) such that
    \begin{enumerate}
        \item If \( a, b \in J \), then \( a \lor b \in J \).
        \item If \( a \in J \) and \( b \leq a \), then \( b \in J \).
    \end{enumerate}
    The set of all ideals of \( L \) is denoted \( \Id L \), which forms a lattice with respect to set inclusion.
    A special class of ideals are the \emph{principal ideals}, which are of the form \( \pid{a} = \setc{b \in L}{b \leq a} \). We will sometimes use the notation \( \pid{a}_L \) to emphasize that this is the principal ideal of \( a \) inside \( L \).
\end{defn}

\begin{prop}
    For a lattice \( L \), the principal ideal map
    \begin{align*}
        I : \, &L \to \Id L \\
        &a \mapsto (a]
    \end{align*}
    is an injective homomorphism of lattices.
\end{prop}

\begin{proof}
    See Gr\"{a}tzer \cite[Corollary 4, p.~24]{gratzer1998}.
\end{proof}

\begin{thm}
    \label{thm:maddux}
    For a modular lattice \( L \) with least element 0, the set of reflexive equivalence elements \( \mlat{\alg{L}} \) and the set of ideals \( \Id L \) are identical as lattices. That is, \( \mlat{\alg{L}} = \Id L \), and for all ideals \( J, K \in \Id L \)
    \begin{align*}
        J \lor K &= J \circ K \\
        J \land K &= J \cap K.
    \end{align*}
\end{thm}
\begin{proof}
    See Maddux \cite[p.~243]{maddux1981}.
\end{proof}

\section{Embeddings of Lattice Complex Algebras}

The purpose of this section is to prove Theorem \ref{thm:embed}, which states that for all complete sublattices \( K \) of a modular lattice \( L \), there is a corresponding complete embedding \( \phi : \alg{K} \to \alg{L} \) of Boolean monoids. The construction of this map will rely on the following result from the theory of relation algebras.

\begin{thm}
    \label{thm:aet}
    Let \( A \) be a complete and atomic Boolean monoid, \( U \) the set of atoms of \( A \), and \( B \) a complete Boolean monoid. Suppose \( \phi : U \to B \) is a map with the following properties:
    \begin{enumerate}
        \item The elements \( \phi(u) \) for \( u \in U \) are non-zero, mutually disjoint, and have a join of 1 in \( B \).
        \item \( t = \bigvee \setc{\phi(u)}{u \in U \text{ and  } u \leq t} \)
        \item \( \phi(u) \circ \phi(v) = \bigvee \setc{\phi(w)}{w \in U \text{ and  } w \leq u \circ v} \) for all \( u, v \in U \).
    \end{enumerate}
    Then \( \phi \) extends in a unique way to a complete embedding \( \phi : A \to B \) of Boolean monoids, given by
    \begin{equation*}
    \phi(r) = \bigvee \setc{\phi(u)}{u \in U \text{ and } u \leq r },
    \end{equation*}
    where \( r \) is any element of \( A \).
\end{thm}
\begin{proof}
    This is a specialization of Givant \cite[Theorem 7.13 and Corollary 7.14]{introrel} to Boolean monoids.
\end{proof}

We apply this to lattice complex algebras as follows. Recall that for a complete lattice \( L \), a subset \( K \subseteq L \) is a \emph{complete sublattice} iff for all \( S \subseteq K \), \( \bigwedge S \in K \) and \( \bigvee S \in K \), where these infima and suprema are calculated in \( L \).

\begin{thm}
    \label{thm:embed}
    Let \( L \) be a complete modular lattice, \( K \subseteq L \) a complete sublattice, and \( I_K \) and \( I_L \) their respective principal ideal maps. Then there is a unique complete embedding of Boolean monoids \( \phi : \alg{K} \to \alg{L} \) such that \( \phi \circ I_K = I_L \).
\end{thm}

\begin{proof}
    We will first show uniqueness to determine what the map \( \phi \) should be, and then use that definition to show it is a complete embedding.

    Suppose that \( \phi : \alg{K} \to \alg{L} \) is a complete embedding with \( \phi \circ I_K = I_L \). Since \( \phi \) is a complete homomorphism, it is determined by its values on the singleton subsets of \( K \), which are the atoms of \( \alg{K} \). For all \( a \in K \), we can write \( \pid{a}_K \) as the disjoint union
    \begin{alignat*}{2}
        \pid{a}_K &= \braces*{a} \cup \bigcup_{\substack{b \, < \, a \\ b \, \in \, K}} \pid{b}_K \\
        \implies \phi \parens{\pid{a}_K} &= \phi \parens*{\braces*{a}} \cup \bigcup_{\substack{b \, < \, a \\ b \, \in \, K}} \phi \parens*{\pid{b}_K} \qquad && \phi \text{ is a complete homomorphism} \\
        \implies \pid{a}_L &= \phi \parens*{\braces*{a}} \cup \bigcup_{\substack{b \, < \, a \\ b \, \in \, K}} \pid{b}_L && \phi \circ I_K = I_L \\
        \implies \phi \parens*{\braces*{a}} &= \pid{a}_L \mathbin{\big\backslash} \bigcup_{\substack{b \, <  \, a \\ b \, \in \, K}} \pid{b}_L && \phi \text{ preserves disjoint unions}
        \end{alignat*}
        Thus \( \phi \) is uniquely determined.

        So then, let \( U \) be the set of singletons in \( \alg{K} \), and define the map \( \phi : U \to \alg{L} \) by
        \begin{equation*}
            \phi \parens*{\braces*{a}} = \pid{a}_L \mathbin{\big\backslash} \bigcup_{\substack{b \, <  \, a \\ b \, \in \, K}} \pid{b}_L
        \end{equation*}
        We will verify the three conditions of Theorem \ref{thm:aet} to show that this can be extended to a complete embedding of Boolean monoids.
    \begin{enumerate}
        %----------- 1
        \item It suffices to show that the sets \( \phi \parens*{\braces*{a}} \) for \( a \in K \) are non-empty, mutually disjoint, and cover \( L \).
    \begin{itemize}
        \item All sets are non-empty, since \( a \in \phi \parens*{\braces*{a}} \) for any \( a \in K \).
        \item Let \( a, b \in K \) be distinct elements. Then \( a \land b \leq a \), and \( a \land b \leq b \). Since \( a \) and \( b \) are distinct, at least one of the previous inequalities must be strict, so without loss of generality suppose \( a \land b < a \). Since \( K \) is a sublattice, \( a \land b \in K \).
            Now suppose \( x \in \phi \parens*{\braces*{a}} \cap \phi \parens*{\braces*{b}} \). Then \( x \leq a \) and \( x \leq b \), so \( x \in \pid{a \land b}_{L} \). Since \( a \land b < a \), this implies \( x \notin \phi \parens*{\braces*{a}} \), which is a contradiction. Thus \( \phi \parens*{\braces*{a}} \) and \( \phi \parens*{\braces*{b}} \) are disjoint.

    \item For an arbitrary \( x \in L \), let
        \begin{equation*}
            a = \bigwedge F_x
        \end{equation*}
        where \( F_x = \setc{b \in K}{x \leq b} \). Since \( K \) is a complete sublattice, this infimum exists and is an element of \( K \). By definition, \( x \) is a lower bound for \( F_x \), so \( x \leq a \), and thus \( x \in \pid{a}_L \). Since \( a \in F_x \), we in fact have \( a = \min F_x \). Furthermore, for any other \( b \in K \) with \( b < a \), it cannot be that \( x \in \pid{b}_L \), since then we would have \( a \leq b \), which is impossible. Thus
    \begin{align*}
        x \in \pid{a}_L \mathbin{\big\backslash} \bigcup_{\substack{b \, < \, a \\ b \, \in \, K}} \pid{b}_L = \phi \parens*{\braces*{a}}.
    \end{align*}
    So the images of \( \phi \) cover \( L \).
    \end{itemize}
    %-------------- 2
    \item The monoid identity \( t = \braces*{0} \) is itself a singleton, and from the definition of \( \phi \) we trivially have
    \begin{equation*}
        \phi \parens*{t} = \phi \parens*{\braces*{0}} = \pid{0}_L = \braces*{0} = t.
    \end{equation*}
    %-------------- 3
    \item From left to right, let \( a, b \in K \), and suppose that \( z \in \phi \parens*{\braces*{a}} \circ \phi \parens*{\braces*{b}} \). We wish to show \( z \in \phi \parens*{\braces*{c}} \), for some \( c \in K \) with \( \braces*{c} \subseteq \braces*{a} \circ \braces*{b} \).

    By assumption \( Rxyz \) for some \( x \in \phi \parens*{\braces*{a}} \) and \( y \in \phi \parens*{\braces*{b}} \). From the first condition, we know \( a = \min F_x \), \( b = \min F_y \), and \( z \in \phi \parens*{\braces*{c}} \), where \( c = \min F_z \). Since \( x \leq a \) and \( y \leq b \), we have
    \begin{alignat*}{2}
        x \lor y &\leq a \lor b \\
        \implies x \lor z &\leq a \lor b \qquad && \text{since } Rxyz \\
        \implies z &\leq a \lor b \\
        \implies c &\leq a \lor b && \text{minimality of } c \\
        \implies a \lor c &\leq a \lor b.
    \end{alignat*}
    Symmetrically, we conclude \( a \lor b \leq a \lor c \), and so \( a \lor b = a \lor c \). A similar argument shows \( a \lor c = b \lor c \). Thus \( Rabc \), and so \( c \in \braces{a} \circ \braces{b} \) as desired.

    From right to left, let \( a, b, c \in K \) and suppose \( \braces{c} \subseteq \braces{a} \circ \braces{b} \). We wish to show \( \phi \parens*{\braces*{c}} \subseteq \phi \parens*{\braces*{a}} \circ \phi \parens*{\braces*{b}} \). That is, for all \( z \in \phi \parens*{\braces*{c}} \), there exists \( x \in \phi \parens*{\braces*{a}} \) and \( y \in \phi \parens*{\braces*{b}} \) such that \( Rxyz \). To do this, we use an approach similar to the one of Maddux in \cite[p.~244]{maddux1981}. For a given \( z \), let
    \begin{align*}
        x &= (b \lor z) \land a \\
        y &= (a \lor z) \land b.
    \end{align*}
            We first show that \( a = \min F_x \). From the definition of \( x \) we have \( a \land b \leq x \leq a \), so \( a \in F_x \). Now let \( d \in F_x \) be any other element. Then \( d \in K \) with \( x \leq d \), so \( x \leq a \land d \). Furthermore,
    \begin{alignat*}{2}
        x \lor b &= ((b \lor z) \land a) \lor b \qquad && \text{definition of } x\\
                 &= (b \lor z) \land (a \lor b) && \text{modularity} \\
                 &= (b \lor z) \land (b \lor c) && \text{since } Rabc \\
                 &= b \lor z && \text{since } z \leq c.
    \end{alignat*}
    Since \( z \leq b \lor z = x \lor b \), we then have
    \begin{alignat*}{2}
        z &\leq (a \land d) \lor b \qquad && \text{since } x \leq a \land d \\
        \implies c &\leq (a \land d) \lor b && \text{minimality of } c \\
        \implies b \lor c &\leq (a \land d) \lor b \\
        \implies a \lor b &\leq (a \land d) \lor b && \text{since } Rabc.
    \end{alignat*}
    Using absorption, this implies
    \begin{alignat*}{2}
     a &\leq ((a \land d) \lor b) \land a \\
       &= (a \land d) \lor (a \land b) \qquad && \text{modularity} \\
       &= a \land d && \text{since } a \land b \leq x \leq a \land d.
    \end{alignat*}
    Therefore \( a \leq d \), so \( a = \min F_x \) as wanted. Thus \( x \in \phi \parens*{\braces*{a}} \), and a symmetric argument shows that \( y \in \phi \parens*{\braces*{b}} \).

    Now we show \( Rxyz \). Using modularity,
    \begin{align*}
        x \lor z &= ((b \lor z) \land a) \lor z \\
                 &= (b \lor z) \land (a \lor z) \\
                 &= (a \lor z) \land (b \lor z) \\
                 &= ((a \lor z) \land b) \lor z \\
                 &= y \lor z.
    \end{align*}
    Since \( x \leq a \) and \( z \leq c \), we have \( x \lor z \leq a \lor c = a \lor b \). Thus
    \begin{alignat*}{2}
        x \lor z &= (a \lor b) \land (x \lor z) \\
                 &= (a \lor b) \land (a \lor z) \land (b \lor z) && \text{from above} \\
                 &= (a \lor (b \land (a \lor z))) \land (b \lor z) \qquad && \text{modularity} \\
                 &= (b \lor z) \land (a \lor ((a \lor z) \land b)).
    \end{alignat*}
    Using that \( (a \lor z) \land b \leq b \leq b \lor z \) and a final application of the modular law, we thus have
    \begin{align*}
        x \lor z &= ((b \lor z) \land a) \lor ((a \lor z) \land b) \\
                 &= x \lor y.
    \end{align*}
    Thus \( Rxyz \), and the condition is shown.
    \end{enumerate}
    % ------------- composition
    Thus by Theorem \ref{thm:aet}, \( \phi \) extends uniquely to a complete embedding \( \phi : \alg{K} \to \alg{L} \) of Boolean monoids, where for all \( S \subseteq K \),
    \begin{equation*}
    \phi \parens*{S} = \bigcup_{a \, \in \, S} \phi \parens*{\braces*{a}}.
    \end{equation*}
    We use this definition to show that \( \phi \circ I_K = I_L \). For any \( a \in K \), \( \pid{a}_K \) is a reflexive equivalence element of \( \alg{K} \) by Theorem \ref{thm:maddux}. Since \( \phi \) preserves equational properties, the image \( \phi \parens*{\pid{a}_K} \) is also a reflexive equivalence element of \( \alg{L} \), and thus an ideal of \( L \) by the same theorem. From the definition of \( \phi \),
    \begin{equation*}
        a \in \phi \parens*{\braces*{a}} \subseteq \phi \parens*{\pid{a}_K},
    \end{equation*}
    and so \( \pid{a}_L \subseteq \phi \parens*{\pid{a}_K} \) from the definition of an ideal.

    On the other hand,
    \begin{equation*}
        \phi \parens*{\pid{a}_K} = \bigcup_{\substack{b \, \leq \, a \\ b \, \in \, K}} \phi \parens*{\braces*{b}}
                                \subseteq \bigcup_{\substack{b \, \leq \, a \\ b \, \in \, K}} \pid{b}_L
                                = \pid{a}_L,
    \end{equation*}
    and so \( \phi \parens*{\pid{a}_K} = \pid{a}_L \).
\end{proof}

\section{An Epimorphism That is Not Surjective}

In this section, let \( \mathsf{LRA} \) be the class of all subalgebras of lattice complex algebras, \( \mathsf{ARA} \) the variety of abelian relation algebras, and \( \mathsf{R} \) any class of relation algebras with \( \mathsf{LRA} \subseteq \mathsf{R} \subseteq \mathsf{ARA} \). We will now use the following general construction and modular lattices constructed by Freese to show that ES fails for \( \mathsf{R} \).

Let \( L \) be a complete modular lattice, and \( K \subseteq L \) a complete sublattice. The principal ideal map \( I_L : L \to \Id L \) is an embedding of lattices, and from Theorem \ref{thm:maddux} we know \( \Id L = \mlat{\alg{L}} \). Thus, let \( K' = I_L(K) \) and \( L' = I_L(L) \) be the isomorphic images of \( L \) and \( K \) contained in \( \mlat{\alg{L}} \). In \( \alg{L} \), let \( U \) be the subalgebra generated by \( K' \), and \( V \) the subalgebra generated by \( L' \).
\begin{thm}
    \label{thm:episurj}
    In the above situation, if \( K \) is a proper epic sublattice of \( L \), then \( U \) is a proper \( \mathsf{R} \)-epic subalgebra of \( V \).
\end{thm}
\begin{proof}
    Since \( K' \subset L' \) we have \( U \subseteq V \). Let \( W \) be any other algebra of \( \mathsf{R} \), and \( f, g : V \to W \) two homomorphisms that agree on \( U \). The image of a reflexive equivalence element is a reflexive equivalence element, so \( f \) and \( g \) restrict to maps
    \begin{equation*}
        f |_{L'}, \, g|_{L'} : L' \to \mlat{W}.
    \end{equation*}
    By Theorem \ref{thm:modularlat}, \( \mlat{W} \) is a modular lattice under fusion and meet, and \( f \) and \( g \) preserve these operations, so these restrictions are homomorphisms of modular lattices. By assumption, \( f \) and \( g \) agree on \( U \), and since \( K' \subseteq U \) they must also agree on \( K' \). But \( K' \) is an epic sublattice of \( L' \), so \( f \) and \( g \) must also agree on \( L' \). Thus \( f|_{L'} = g|_{L'} \), and so \( f = g \) since \( L' \) is the generating set of \( V \). Thus \( U \) is an \( \mathsf{R} \)-epic subalgebra of \( V \).

    However, \( U \) is a proper subalgebra. Let \( \phi : \alg{K} \to \alg{L} \) be the complete embedding of Theorem \ref{thm:embed}, and let \( Z = \im \phi \). Since \( I_K(K) \subseteq \alg{K} \), we have
    \begin{equation*}
        \phi(I_K(K)) = I_L(K) = K',
    \end{equation*}
    so \( K' \subseteq Z \), which implies \( U \subseteq Z \) since \( U \) is the smallest subalgebra that contains \( K' \). Now for contradiction suppose that \( U = V \). Then \( L' \subseteq V = U \subseteq Z \). Thus for any \( x \in L \), we have \( \pid{x}_L \in Z \), so there is some \( S \subseteq K \) such that \( \phi(S) = \pid{x}_L \). In particular then, there is an \( a \in S \) such that
    \begin{equation*}
        x \in \phi \parens*{\braces*{a}} \subseteq \pid{a}_L \implies x \leq a.
    \end{equation*}
    On the other hand,
    \begin{equation*}
        a \in \phi \parens*{\braces*{a}} \subseteq \phi \parens*{S} = \pid{x}_L \implies a \leq x.
    \end{equation*}
    Thus \( x = a \), so \( x \in K \). But the element \( x \in L \) was arbitrary, so \( L = K \), which is a contradiction.
\end{proof}

\begin{thm}
ES fails for any class \( \mathsf{R} \) of relation algebras with \( \mathsf{LRA} \subseteq \mathsf{R} \subseteq \mathsf{ARA} \).
\end{thm}

\begin{proof}
    In \cite[Theorem 3.3]{freese1979}, Freese constructs modular lattices \( A \subset B \) such that \( A \) is a proper epic sublattice of \( B \). \( B \) has no infinite chains, so is complete by Davey and Priestley \cite[Theorem 2.41 (iii)]{davey2002}. Likewise, \( A \) is a \{0,1\}-sublattice of \( B \), and as a sublattice is complete by Theorems 2.40 and 2.41 (i) of the same. We can thus apply Theorem \ref{thm:episurj} to \( A \) and \( B \), and the result follows.
\end{proof}

\begin{cor}
    ES fails for the varieties of abelian, symmetric, and dense symmetric relation algebras.
\end{cor}

\begin{cor}
    The Beth Definability Property fails for \textbf{KR}.
\end{cor}

\section{Conclusion}

Using modular lattices constructed by Freese, we have shown that epimorphisms need not be surjective in a wide class of relation algebras. This class includes the Boolean monoids, which shows that the Beth Definability Property fails for the super relevant logic \textbf{KR}. This should be contrasted with the result of \cite[Theorem 8.5]{bezhanishvili2017}, which shows that the Beth Definability Property does hold for the super relevant logic \textbf{RM}. The super relevant logics thus exhibit more diversity than the relevant logics, where this property fails uniformly.

\subsection*{Acknowledgements}

The author would like to thank Katalin Bimb\'{o} for suggesting this problem for his undergraduate research project and supervising his work on it. Her advice and never-ending encouragement were instrumental in finding a solution.

\bibliography{bethkr}
\bibliographystyle{plain}
%\printbibliography

\end{document}